\documentclass[12pt]{amsart}

\usepackage{mathtools}
\usepackage{amssymb}
\usepackage{hyperref}
\usepackage{fullpage}

\usepackage{tikz}
\usetikzlibrary{decorations.pathreplacing}
\tikzset{mybrace/.style={decoration={brace,raise=1.8mm},decorate}}
\tikzset{mybracedown/.style={decoration={brace,mirror,raise=1.8mm},decorate}}
\usetikzlibrary{math}

\newcounter{thm} \numberwithin{thm}{section}
\newtheorem{theorem}[thm]{Theorem}

\newtheorem{lemma}[thm]{Lemma}

\newtheorem{definition}[thm]{Definition}
 
\newtheorem*{claim}{Claim}

\renewcommand{\bar}[1]{\overline{#1}}
\newcommand{\R}{\mathbb{R}}

\makeatletter
\DeclareRobustCommand*\uell{\mathpalette\@uell\relax}
\newcommand*\@uell[2]{
  \setbox0=\hbox{$#1\ell$}
  \setbox1=\hbox{\rotatebox{12}{$#1\ell$}}
  \dimen0=\wd0 \advance\dimen0 by -\wd1 \divide\dimen0 by 2
  \mathord{\lower 0.1ex \hbox{\kern\dimen0\unhbox1\kern\dimen0}}
}

\renewcommand{\ln}[0]{\operatorname{\uell n}}

\title{Convexity, Squeezing, and the Elekes-Szab\'{o} Theorem}

\date{January 1, 2024}

\author[O. Roche-Newton]{Oliver Roche-Newton} 
\address{Insitute for Algebra, Johannes Kepler Universit\"{a}t, Linz, Austria}
\email{o.rochenewton@gmail.com}

\author[E. Wong]{Elaine Wong}
\address{Comp. Sci. \& Math Division, Oak Ridge National Laboratory, Oak Ridge, TN, USA}
\email{wongey@ornl.gov}

\begin{document}

\maketitle

%%%%%%%%%%%%%%%%%%%%%
\begin{abstract}

This paper explores the relationship between convexity and sum sets. In particular, we show that elementary number theoretical methods, principally the application of a squeezing principle, can be augmented with the Elekes-Szab\'{o} Theorem in order to give new information. Namely, if we let $A \subset \mathbb R$, we prove that there exist $a,a' \in A$ such that
\[
\left | \frac{(aA+1)^{(2)}(a'A+1)^{(2)}}{(aA+1)^{(2)}(a'A+1)} \right | \gtrsim |A|^{31/12}.
\]
We are also able to prove that
\[
\max \{|A+ A-A|, |A^2+A^2-A^2|, |A^3 + A^3 - A^3|\} \gtrsim |A|^{19/12}.
\]
Both of these bounds are improvements of recent results and takes advantage of computer algebra to tackle some of the computations.

\end{abstract}

%%%%%%%%%%%%%%%%%%%%%
\section{Introduction}

We choose to begin with some boring stuff up front in order to avoid any awkwardness between us that may originate from cultural differences in mathematical notation. Throughout this paper\footnote{This is the authors' version of the work. It is posted here for your personal use. Not for redistribution. The definitive version was published in the {Electronic Journal of Combinatorics}, Volume~31, Issue~1, Article Number P1.3 (2024). \url{https://doi.org/10.37236/11331}.}, the notation $X\gg Y$, $Y \ll X,$ $X=\Omega(Y)$, and $Y=O(X)$ are all equivalent and mean that $X\geq cY$ for some absolute constant $c>0$. $X \approx Y$ and $X=\Theta (Y)$ denote that both $X \gg Y$ and $X \ll Y$ hold. $X \gg_a Y$ means that the implied constant is no longer absolute, but depends on $a$. We also use the notation $X \gtrsim Y $ and $Y \lesssim X$ to denote that $X \gg Y/(\log Y)^c$ for some absolute constant $c>0$. Unless otherwise stated, all logarithms are in base 2 and we will use the notation $\ln$ for logs in base $e$. 

Now for the fun stuff. An important generalization of the sum-product phenomenon is the idea that strictly convex or concave functions destroy additive structure.\footnote{For the sake of completeness and conciseness, it can be assumed that a convexity result has an analogous concavity result and we will only mention the former in statements of theorems and in the proofs.} For instance, Elekes, Nathanson, and Ruzsa \cite{ENR} used incidence geometry to prove that the bound
\begin{equation*} \label{ENR}
|A+A|^2 |f(A)+f(A)|^2 \gg  |A|^5
\end{equation*}
holds for any $A\subset \mathbb R$ and any strictly convex function $f$.

A recent trend in this area of sum-product theory has seen elementary methods play a more prominent role. These elementary methods have their origins in the work of Ruzsa, Shakan, Solymosi, and Szemer\'{e}di \cite{RSSS}, who gave a simple and beautiful proof of the fact that
\begin{equation} \label{RSSS}
|A+A-A| \gg |A|^2
\end{equation}
holds for any convex set $A \subset \mathbb R$. A set $A$ is said to be (\textit{strictly}) \textit{convex} if it has strictly increasing consecutive differences. That is, if we write $A=\{a_1<a_2< \dots <a_n \}$ then $a_i-a_{i-1} < a_{i+1} - a_i$  holds for all $2 \leq i \leq n-1$. Note that the bound \eqref{RSSS} is optimal up to constant factors, as can be seen by taking $A=\{1,4,\dots,n^2 \}$.

For the purposes of our work, we will apply an elementary ``squeezing" method to two different number-theoretic problems. This method was the foundation for the proof of~\eqref{RSSS}, and has already been further developed in a series of recent papers (see \cite{B}, \cite{BHR}, \cite{HRNR}, \cite{HRNS}, \cite{M}). 

\subsection{Products and Shifts}
The squeezing method was used in~\cite{HRNS} to give lower bounds for expanders involving convex and `superconvex' functions. Informally, a superconvex function is a strictly convex function which has a strictly convex first derivative (this notion was considered in much more detail in~\cite{HRNR}). The following result was established in \cite{HRNS}: for any $X \subset \mathbb R$, there exist $x,x' \in X$ such that
\begin{equation} \label{shiftsold}
\left | \frac{(xX+1)^{(2)}(x'X+1)^{(2)}}{(xX+1)^{(2)}(x'X+1)} \right | \gtrsim |X|^{5/2}.
\end{equation}
The notation $U^{(k)}$ is used for the $k$-fold product set of $U$, that is, 
\[
U^{(k)}:=\{u_1\cdots u_k\colon u_1,\ldots,u_k\in U\}.
\]
It is perhaps not obvious why the bound \eqref{shiftsold} has anything to do with convexity. However, this is in fact a special case of a more general result involving growth of the set $f(a+A)$ under addition, see \cite[Theorem 2.6]{HRNS}. Then the bound \eqref{shiftsold} follows by considering the strictly convex function $f(x)=\ln(e^x+1)$ and $A= \ln X$. In this paper, we give an improved bound for this expander.

\begin{theorem} \label{thm:main2}
For any finite set $X \subset \mathbb R$, there exists $x,x' \in X$ such that
\begin{equation*} \label{shiftsnew}
\left | \frac{(xX+1)^{(2)}(x'X+1)^{(2)}}{(xX+1)^{(2)}(x'X+1)} \right | \gtrsim |X|^{31/12}.
\end{equation*}
\end{theorem}

The primary motivation for proving the bound \eqref{shiftsold} in \cite{HRNS} was that it could be used to give a better non-trivial lower bound for the number of dot products determined by a point set in the Euclidean plane. By using Theorem \ref{thm:main2} in place of \eqref{shiftsold} in the analysis of~\cite{HRNS}, one can obtain a further quantitative improvement to the main result of \cite{HRNS}. However, we will not explore this in detail here. Rather, we choose to focus on highlighting the methods that allow us to improve the bound \eqref{shiftsold}.

\subsection{Two Convex Functions}

In \cite{HRNR}, the same squeezing technique was used to prove that
\begin{equation} \label{HRNR}
|A+A-A||f(A)+f(A)-f(A)| \gg \frac{|A|^{3}}{ \log |A| }
\end{equation}
holds for any $A \subset \mathbb R$ and any strictly convex function $f$. In particular, it follows that
\begin{equation} \label{max}
\max \{ |A+A-A| , |f(A) + f(A) - f(A)| \} \gtrsim |A|^{3/2}.
\end{equation} 
The bound \eqref{HRNR} is also tight, up to constant and logarithmic factors, as is illustrated by the case when $A= [n]$ for some positive integer $n$ and $f(x)=x^2$. Moreover, a recent paper of Bradshaw \cite{B} removed the logarithmic factors from \eqref{HRNR}.

An interesting open problem is to give a version of \eqref{max} with an exponent strictly greater than $3/2$. Although we are unable to give such a result, we do obtain an improvement by instead simultaneously considering two different convex functions.

\begin{theorem} \label{thm:main1}
For any $A \subset \mathbb R$,
\[
\max \{|A+A-A|, |A^2+A^2-A^2|, |A^3+A^3-A^3| \} \gtrsim |A|^{19/12}.
\]
\end{theorem}

%%%%%%%%%%%%%%%%%%%%
\section{The Elekes-Szab\'{o} Theorem}

We will prove Theorems \ref{thm:main2} and \ref{thm:main1} in the next two sections using a ``squeezing'' idea originating from \cite{RSSS}. However, we would first like to introduce an additional tool in the form of the Elekes-Szab\'{o} Theorem in order to make new quantitative progress with these problems. The Elekes-Szab\'{o} Theorem gives a bound for the size of the intersection of a Cartesian product and a suitably non-degenerate algebraic surface. The first non-trivial bound for this problem was given by Elekes and Szab\'{o} in \cite{ES}. We will use the most recent quantitative version of this result, which is due to Solymosi and Zahl \cite{SZ}, building on previous work of Raz, Sharir, and de Zeeuw \cite{RSdZ}.

\begin{theorem} \label{thm:11/6}
Let $F \in \R[x,y,z]$ be a non-degenerate irreducible polynomial of degree $d$, such that none of the partial derivatives $\frac{\partial F}{\partial x}, \frac{\partial F}{\partial y}, \frac{\partial F}{\partial z}$ vanish. Then, for all $A,B,C\subseteq \R$ such that $|A| \leq |B| \leq |C|$,
\begin{equation}\label{eq:RSZ11/6}
|Z(F) \cap (A \times B \times C)| = O_d((|A||B||C|)^{4/7}+|B ||C|^{1/2}).
\end{equation}
In particular,
\[
|Z(F) \cap (A \times B \times C)| = O_d( \max \{|A|,|B|,|C| \}^{12/7} ).
\]
\end{theorem}

We are often interested in the case when $|A|=|B|=|C|=n$, in which case \eqref{eq:RSZ11/6} simplifies to become
\[
|Z(F) \cap (A \times B \times C)| \ll_d n^{12/7}.
\]
On the other hand, a simple construction of \cite{MRNWdZ} shows the existence of a non-degenerate quadratic polynomial $F: \mathbb R^3 \rightarrow \mathbb R$ and a set $A$ of cardinality $n$ with
\[
|Z(F) \cap (A \times A \times A)| \gg n^{3/2}.
\]

Of course, to understand Theorem \ref{thm:11/6}, one needs to know what it means for $f$ to be non-degenerate. 

\begin{definition}
\label{def:nondegen} A polynomial $F \in \R[x,y,z]$ is \textit{degenerate} if there exists a one-dimensional sub-variety $Z_0 \subseteq Z(F)$ such that for all $v\in Z(F)\setminus Z_0$, there are open intervals $I_1,I_2,I_3\subseteq \R$ and injective real-analytic functions $\phi_i:I_i\rightarrow \R$ with real-analytic inverses ($i = 1,2,3$) such that $ v \in I_1\times I_2\times I_3$ and for any $(x,y,z)\in I_1\times I_2\times I_3$, we have
\[ (x, y, z) \in Z(F) \text{ if and only if } \phi_1(x) + \phi_2(y) + \phi_3(z) = 0\,.\]
Otherwise, we say that $F$ is \textit{non-degenerate}.
\end{definition}

This definition is rather technical. However, it may be helpful for the reader to consider an example of a function for which Theorem \ref{thm:11/6} cannot hold (and which therefore must be degenerate). If we take $F(x,y,z)=x+y-z$ and $A=B=C=[n]$, then it is not difficult to check that 
\[
| \{ (a,b,c) \in A \times B \times C : F(a,b,c)=0 \}| \gg n^{2}.
\]
In this case, the surface $Z(F)$ is a hyperplane. Theorem \ref{thm:11/6} therefore tells us that this $F$ must be degenerate. We can also see this by looking at Definition \ref{def:nondegen} directly; here we set $\phi_1(t)= \phi_2(t)=t$ and $\phi_3(t)=-t$. This construction can be modified in many ways. For instance, one can take $F(x,y,z)=x^2+y^2-z^2$ and $A=B=C=\{ \sqrt 1, \sqrt 2, \dots, \sqrt n \}$. Or one can take $F(x,y,z)=xyz$ and $A=B=C= \{ 2^i : i \in \pm [n] \}$. These examples are based on the same common idea, that the surface is a kind of deformation of a hyperplane, and this is captured conveniently by Definition \ref{def:nondegen}. 

For the case when the expression $F(x,y,z)=0$ can already be rearranged into the form $z=f(x,y)$ where $f$ is a polynomial, the definition of non-degeneracy becomes more straightforward (see \cite[Theorem 2]{RSS}). However, for the general case, it is not always immediately obvious whether a given polynomial is degenerate or non-degenerate. To help with this task, we will use an idea introduced by Elekes and R\'onyai, which is that non-degeneracy can be verified using the following derivative test~\cite{ER00}. We include a proof here for completeness, which is essentially taken from \cite[Lemma~33]{ES}. 

\begin{lemma}\label{lem:difftest}
Let $f:\R^2\to \R$ be a bivariate twice-differentiable real function and let $U$ be a nonempty open subset of $\emph{Dom}(f)\setminus(\lbrace f_y=0\rbrace \cup\lbrace f_x=0\rbrace).$ Suppose that there exist (twice-differentiable) univariate real functions $\psi, \varphi_1,\varphi_2$ such that
\[
 f(x,y) = \psi(\varphi_1(x) + \varphi_2(y)).
\]
Then
\begin{equation}\label{eq:test}
\frac{\partial^2\left(\ln|f_x/f_y|\right)}{\partial x\partial y}
\end{equation}
is identically zero on $U$.
\end{lemma}

\begin{proof}
First, we can calculate the partial derivatives $f_x$ and $f_y$ using the chain rule to get
\begin{align*}
f_x &=\varphi_1'(x) \cdot \psi'( \varphi_1(x) + \varphi_2(y)),
\\ f_y &=\varphi_2'(y) \cdot \psi'( \varphi_1(x) + \varphi_2(y)).
\end{align*}
Therefore, 
\[
\frac{f_x}{f_y} = \frac{\varphi_1'(x) }{\varphi_2'(y) }.
\]
We can then calculate the form of \eqref{eq:test}. Our assumptions on the set $U$ ensure that this expression can be evaluated for all $(x,y) \in U$. First, we can use the chain rule to calculate the partial derivative with respect to $x$. We obtain
\[
\frac{\partial\left(\ln|f_x/f_y|\right)}{\partial x} = \frac{f_y}{f_x} \cdot \frac{ \varphi_1''(x)}{\varphi_2'(y)} = \frac{ \varphi_1''(x)}{\varphi_1'(x)}.
\]
This function does not depend on $y$, and so differentiating with respect to $y$ gives a function which is identically zero. We conclude that
\[ 
\frac{\partial^2\left(\ln|f_x/f_y|\right)}{\partial x\partial y} (x,y)=0
\]
for all $(x,y) \in U$, as required.
\end{proof}

In practice, Lemma \ref{lem:difftest} allows us to test the whether or not the polynomial $F(x,y,z)$ is degenerate by rearranging the expression $F(x,y,z)=0$ into the form $z=f(x,y)$, computing the resulting expression \eqref{eq:test}, and checking whether it is identically zero on an open set of $\R^2$. This characterization will be used repeatedly, and so we record it in the form of the next lemma. The result is implicit in earlier work going back to \cite{ER00}, but since we are not aware of it being stated in this form, we give the full statement and its proof here.

\begin{lemma} \label{lem:test}
Suppose that $F: \mathbb R^3 \rightarrow \mathbb R$ is a degenerate polynomial. Let $f(x,y)$ be a twice differentiable real function with $f_x$ and $f_y$ not identically zero such that
\begin{equation} \label{iff1}
z=f(x,y) \Leftrightarrow F(x,y,z)=0.
\end{equation}
Then there exists a nonempty open set $U \subset \mathbb R^2$ such that
\[
\frac{\partial^2\left(\ln|f_x/f_y|\right)}{\partial x\partial y}
\]
is identically zero on $U$.
\end{lemma}

\begin{proof}

Since $F$ is degenerate, there is some open neighborhood $I_1\times I_2\times I_3$ intersecting $Z(F)$, and some smooth functions $\varphi_1$, $\varphi_2$ and $\varphi_3$ with smooth inverses, such that, for all $(x,y,z) \in I_1 \times I_2 \times I_3$, 
\begin{equation} \label{iff2}
F(x,y,z) = 0 \Leftrightarrow \varphi_1(x)+\varphi_2(y) + \varphi_3(z) = 0.
\end{equation}
Then, since $\varphi_3$ has a smooth inverse on $I_3$,  we can define a new smooth function $\psi(t)  = \varphi_3^{-1}(-t)$. The second equation in \eqref{iff2} is therefore equivalent  to  $z = \psi(\varphi_1(x) + \varphi_2(y))$. Making use of the hypothesis \eqref{iff1}, it then follows from \eqref{iff2} that,  for all $(x,y,z) \in I_1 \times I_2 \times I_3$, 
\begin{align*} \label{iff3}
z=f(x,y) \Leftrightarrow F(x,y,z) = 0 & \Leftrightarrow \varphi_1(x)+\varphi_2(y) + \varphi_3(z) = 0
\\& \Leftrightarrow z = \psi(\varphi_1(x) + \varphi_2(y)).
\end{align*}
Therefore, for all $(x,y,z) \in I_1 \times I_2 \times I_3$, we have
\[
f(x,y)=\psi(\varphi_1(x) + \varphi_2(y)).
\]
Define
\[
U:=\{(x,y) \in I_1 \times I_2 : f(x,y) \in I_3 \}\setminus (\{f_x=0\}\cup\{f_y=0\})
\]
and note that $U$ is open, since it is the intersection of three open sets. Also, $U$ is non-empty, since there is some $(x,y,z) \in I_1 \times I_2 \times I_3$ such that $F(x,y,z)=0$ and $f_x$ and $f_y$ are both not identically zero. It then follows from Lemma \ref{lem:difftest} that \eqref{eq:test} is identically zero on $U$.

\end{proof}

The development of the Elekes-Szab\'{o} Theorem has largely been motivated by applications to problems in discrete geometry. See the survey of de Zeeuw \cite{dZ} for more background on this problem and its applications. A recent paper of Roche-Newton \cite{ORN} applied the Elekes-Szab\'{o} Theorem in order to prove new results in sum-product theory. The present paper is strongly influenced by the ideas in \cite{ORN}, which in turn takes ideas from \cite{SS} and \cite{MRNSW}.

%%%%%%%%%%%%%%%%%%%%
\section{Proof of Theorem \ref{thm:main1}}

The proofs of Theorems \ref{thm:main2} and \ref{thm:main1} are similar in structure and are based on the same ideas. However, Theorem \ref{thm:main2} involves an extra level of computational and algebraic difficulty. It is therefore logical to begin by proving Theorem \ref{thm:main1} in order to introduce the main ideas in the easier setting.

In order to prove Theorem \ref{thm:main1}, we will first prove the following result, which may be of independent interest. For any $A,B \subset \mathbb R$ and $G \subset A \times B$, We use the notation
\[
A+_G B:= \{ a+b: (a,b) \in G \}
\]
for the sum set $A+B$ restricted to $G$.

\begin{lemma} \label{lem:main}
For any $A \subset \mathbb R$ and $G \subset A \times A \setminus \{ (a,a) : a \in A \}$,
\begin{equation} \label{SPgraph}
\max \{ |A-_G A|, |A^2-_G A^2|, |A^3-_G A^3| \} \gg |G|^{7/12}.
\end{equation}
\end{lemma}

\begin{proof}
The proof is similar to the proof of the main result of \cite{ORN}. Write
\[
C:=A-_G A,\,\,\,\,D:=A^2-_G  A^2, \,\,\,\, E:=A^3 -_G A^3.
\]
We will double count solutions to the system of equations
\begin{align}
c&=a-b, \label{eq1}
\\ d&=a^2- b^2, \label{eq2}
\\ e&=a^3-b^3, \label{eq3}
\end{align}
such that $(a,b) \in G$ and $(c,d,e) \in C \times D \times E$. Define $S$ to be the number of solutions to this system. A simple observation is that
\begin{equation} \label{easy}
S=|G|
\end{equation}
since each pair $(a,b) \in G$ gives rise to a unique solution.

We seek to find a complementary upper bound for $S$ via an application of Theorem \ref{thm:11/6}. For each contribution to $S$, we can eliminate $a$ and $b$ from the system above. Together, \eqref{eq1} and \eqref{eq2} imply that
\begin{equation*}\label{aandb}
b=\frac{d-c^2}{2c}, \,\,\,\,\,\,\,\,\, a=  \frac{d+c^2}{2c}.
\end{equation*}
Note that we do not need to worry about dividing by zero here, since $G$ does not contain any diagonal pairs with $a=b$, i.e., $c \neq 0$. Substituting this information into \eqref{eq3} and rearranging gives
\begin{equation*} \label{3var}
4ce=3d^2+c^4.
\end{equation*}
Next, define 
\[F(x,y,z):=4xz-3y^2-x^4.
\]
We have thus deduced that every contribution $(a,b,c,d,e)$ to $S$ gives rise to a solution $(c,d,e)$ to the equation $F(c,d,e)=0$. Furthermore, no triple $(c,d,e)$ contributes more than once to~$S$, since for fixed $c$ and $d$ with $c \neq 0$ there exists at most one pair $(a,b)$ which satisfies both \eqref{eq1} and \eqref{eq2}. Therefore, we have
\begin{equation} \label{SandF}
S \leq |Z(F) \cap C \times D \times E|.
\end{equation}

\begin{claim}
$F$ is non-degenerate.
\end{claim}

Assuming that the claim is correct, we can apply Theorem \ref{thm:11/6} to obtain the upper bound
\[
|Z(F) \cap C \times D \times E| \ll  \max \{ |C|,|D|,|E| \}^{12/7}.
\]
Combining this with \eqref{SandF} and \eqref{easy}, we have
\[
|G| \ll \max \{ |C|,|D|,|E| \}^{12/7}
\]
and it follows that
\[
\max \{|C|,|D|,|E|\} \gg |G|^{7/12},
\]
as required. It remains to prove the claim.

\begin{proof}[Proof of Claim]

The expression $F(x,y,z) = 0$ can be rearranged into the form
\[
z =\frac{1}{4} \cdot \frac{3y^2+ x^4}{x}.
\]
We define
\[
f(x,y):=\frac{1}{4}  \cdot \frac{3y^2+ x^4}{x}.
\] 
Suppose for a contradiction that $F$ is degenerate. Then it follows from Lemma \ref{lem:test} that there is an open set $U\subset I_1\times I_2$ on which \eqref{eq:test}
is identically zero. We can calculate \eqref{eq:test} directly and obtain a contradiction. Indeed,
\begin{equation*} \label{calculationagain}
\frac{\partial^2\left(\ln|f_x/f_y|\right)}{\partial x\partial y}=\frac{8x^3y^2}{(x^4-y^4)^2}.
\end{equation*}
This is zero if and only if $(x,y)$ is a point on the coordinate axes other than the origin, that is, if $(x,y)$ belongs to the set
\[
\{(x,0): x \in \mathbb R \setminus \{0\} \} \cup \{ (0,y) : y \in \mathbb R \setminus \{0\} \}.
\]
However, this set does not contain any open subsets in $\R^2$. This contradicts the claimed existence of the set $U$, and thus proves that $F$ is non-degenerate.

\end{proof}

Since the proof of the claim is complete, so is the proof of Lemma \ref{lem:main}.
\end{proof}

Observe that the condition in Lemma \ref{lem:main} that diagonal elements are excluded from $G$ is necessary. Indeed, in the extreme case whereby $G=\{ (a,a) : a \in A \}$, we would have
\[
 A-_G A= A^2-_G A^2 =A^3-_G A^3 = \{ 0 \}
\]
and so the conclusion \eqref{SPgraph} fails in the strongest possible sense.

We will repeatedly apply a basic squeezing argument that was the key building block for the results in \cite{HRNR}. For convenience, we consolidate a form of it here in the following lemma.

\begin{lemma} \label{lem:squeeze}
Let $A \subset \mathbb R$ and write $A=\{a_1<a_2< \dots < a_n \}$. Let $D$ denote the set
\[
D := \{ a_{i+1} - a_i : 1 \leq i \leq n-1 \} 
\]
of consecutive differences determined by $A$.  Suppose that there exists a set $D' \subset D$ and a constant $L$ such that, for all $d \in D'$,
\[
|\{i \in [n-1]\colon  a_{i+1}-a_i = d \}| \geq L.
\]
Then
\[
|A+A-A| \gg |D'|^2L.
\]
\end{lemma}

\begin{proof}
Label the elements of $D'$ in ascending order, so $D'=\{d_1 < d_2 < \dots d_t \}$, where $t=|D'|$. Fix $1 \leq j \leq t$. Since $d_j$ has at least $L$ representations as a consecutive difference, we may consider $L$ distinct and disjoint intervals $(a_i,a_{i+1}]$, each with length $d_j$. Then, for all $ 1 \leq k \leq j$, we have
\[
a_i<a_i+d_k \leq a_{i+1}.
\]
Since $a_i+d_k \in A+A-A$, it follows that there are at least $j$ elements of $A+A-A$ in the half-open interval $(a_i, a_{i+1}]$. Repeating this for each of the $L$ choices for $i$, we see that this $j$ gives rise to at least $jL$ elements of $A+A-A$. Summing over all $j$, it follows that
\[
|A+A-A| \geq \sum_{j=1}^t jL \gg t^2L=|D'|^2L.
\]

\end{proof}

We are now ready to prove Theorem \ref{thm:main1}.

\begin{proof}[Proof of Theorem \ref{thm:main1}]
Let $N=|A|$ and label the elements of $A$ in increasing order:
\[
A=\{ a_1<a_2< \cdots <a_N \}.
\]
Consider the set
\[
D:=\{a_{i+1}-a_i : 1 \leq i \leq N-1 \}
\]
of neighboring differences of elements of $A$. For each $d \in D$ we define 
\[
r(d)=|\{1 \leq i \leq N-1: a_{i+1} -a_i = d \}|.
\]
Note that
\begin{equation} \label{easysum}
\sum_{d \in D} r(d) = N-1.
\end{equation}
We can use a dyadic pigeonholing argument to show that there is some subset $D_1 \subset D$ and some integer $L_1$ such that
\[
|D_1|L_1 \gtrsim N \quad \text{and} \quad \forall d \in D_1, \,\, L_1 < r(d) \leq 2L_1.
\]
Since the dyadic pigeonholing technique is repeatedly used throughout this paper, we will give the full details of its first application below, and will be more succinct later. Indeed, it follows from \eqref{easysum} that
\[
N-1=\sum_{d \in D} r(d) = \sum_{j=0}^{\lceil\log N \rceil} \sum_{d \in D : 2^{j-1} < r(d) \leq 2^j} r(d)
\]
and so there exists some $0 \leq j_0 \leq \lceil \log n \rceil$ such that
\[
  \sum_{d \in D : 2^{j_0-1} < r(d) \leq 2^{j_0}} r(d) \gtrsim N.
\]
Now set
\[
L_1:=2^{j_0-1} \quad \text{and} \quad D_1:=\{d \in D : L_1 < r(d) \leq 2L_1 \}.
\]
Define $H \subset A \times A$ as follows:
\[
H:=\{(a_i,a_{i+1}): 1 \leq i \leq N-1, a_{i+1}-a_i \in D_1\}.
\]
Note that $|H|=\sum\limits_{d\in D_1}r(d) \geq |D_1|L_1 \gtrsim N$. Lemma \ref{lem:squeeze} implies that
\begin{equation} \label{case1}
|A+A-A| \gg |D_1|^2L_1 \gtrsim |D_1|N= |A-_H A| N.
\end{equation}

Now consider the set
\[
A^2-_H A^2= \{a_{i+1}^2 -a_i^2 : (a_i, a_{i+1}) \in H \}.
\]
For $d \in A^2-_H A^2$, define 
\[ 
s(d)=|\{(a_i, a_{i+1}) \in H : a_{i+1}^2 -a_i^2 = d \}.
\]
Note that $\sum_{d \in A^2-_H A^2} s(d)= |H| \gtrsim N$. After dyadic pigeonholing again, we obtain a subset $D_2 \subset A^2-_H A^2$ and some integer $L_2$ such that
\[
|D_2|L_2 \gtrsim N\quad \text{and} \quad \forall d \in D_2, \, \, L_2 < s(d) \leq 2L_2.
\]
Define $H' \subset H$ as follows:
\[
H':=\{(a_i,a_{i+1}) \in H : a_{i+1}^2 -a_i^2 \in D_2 \}.
\]
Note that $|H'|=\sum\limits_{d\in D_2}s(d) \geq |D_2|L_2 \gtrsim N$. Lemma \ref{lem:squeeze} implies that
\begin{equation} \label{case2}
|A^2+A^2-A^2| \gg |D_2|^2L_2 \gtrsim |D_2|N= |A^2-_{H'} A^2| N.
\end{equation}

We repeat this argument one more time. For each $d \in A^3-_{H'}A^3$, define
\[
t(d):= |\{(a_i, a_{i+1}) \in H' : a_{i+1}^3-a_i^3=d \}|.
\]
We have $\sum_{d \in A^3 -_{H'} A^3} t(d) = |H'| \gtrsim N$. Therefore, by a further dyadic pigeonholing step, there exists $D_3 \subset A^3-_{H'} A^3$ and some integer $L_3$ such that
\[
|D_3|L_3 \gtrsim N\quad \text{and} \quad \forall d \in D_3, \, \, L_3 < t(d) \leq 2L_3.
\]
Define $H'' \subset H' \subset H$ as follows:
\[
H'':=\{(a_i,a_{i+1}) \in H' : a_{i+1}^3 -a_i^3 \in D_3 \}.
\]
Note that $|H''|=\sum\limits_{d\in D_3}t(d) \geq |D_3|L_3 \gtrsim N$. Lemma \ref{lem:squeeze} yields
\begin{equation} \label{case3}
|A^3+A^3-A^3| \gg |D_3|^2L_3 \gtrsim |D_3|N= |A^3-_{H''} A^3| N.
\end{equation}
Now apply Lemma \ref{lem:main} to the set $H''$. We have
\begin{align*}
\max\{|A-_{H} A|, |A^2-_{H'} A^2|, |A^3-_{H''} A^3| \} & \geq \max\{|A-_{H''} A|, |A^2-_{H''} A^2|, |A^3-_{H''} A^3| \} 
\\ &\gg |H''|^{7/12} \gtrsim N^{7/12},
\end{align*}
where the first inequality uses the inclusions $H'' \subset H' \subset H$.

However, it then follows from \eqref{case1}, \eqref{case2} and \eqref{case3} that
\begin{align*}
&\max\{|A+A-A|, |A^2+A^2-A^2|, |A^3+A^3-A^3|\} 
\\ &\gtrsim N \cdot \max\{|A-_{H} A|, |A^2-_{H'} A^2|, |A^3-_{H''} A^3| \}
\\& \gtrsim N^{19/12},
\end{align*}
as required.
\end{proof}

\subsection{Possible Generalizations of Theorem \ref{thm:main1}}

Theorem \ref{thm:main1} and its proof suggests a possible generalization, which is that the bound
\begin{equation}\label{gen}
\max \{|A+A-A|, |f(A)+f(A)-f(A)|, |g(A)+g(A)-g(A)| \} \gtrsim |A|^{19/12}
\end{equation}
holds for suitable choices of the functions $f$ and $g$. Theorem \ref{thm:main1} proves \eqref{gen} for the case $f(x)=x^2$ and $g(x)=x^3$. We have also verified that our proof works for the case when $f(x)=x^2$ and $g(x)=x^n$ for all $n \geq 3$. Indeed, for fixed $n\geq 3$, we can replace \eqref{eq3} with the analogous equation $e=a^n-b^n$. The resulting polynomial equation $F(x,y,z)=0$ gives us the corresponding rational function \[\frac{\left(x^2+y\right)^n-\left(y-x^2\right)^n}{2x}.\] Computing the derivative \eqref{eq:test} gives us a more complicated, but nevertheless rational function in $x$ and $y$, namely:
\[
\frac{
\begin{aligned}
& (x^2 + y)^5(x^2-y)^{4n+1}-\left(x^2-y\right)^5(x^2+y)^{4n+1}- 8(n-1)x^2y(y^2-x^4)^{2n+1}(x^4-3y^2)\\
& - 2 (y-x^2)^{3n}(x^2+y)^{n+3}\left((2n-3)x^6+(4n^2-12n+9)x^4 y+(6n-5)x^2y^2-y^3\right)\\
& - 2 (y-x^2)^{n+3} (x^2 + y)^{3 n} \left((2n-3)x^6-(4 n^2-12n+9) x^4 y+(6n-5) x^2 y^2-y^3 \right)
\end{aligned}
}
{
\left(x^4-y^2\right)\left(2 y \left(x^4-y^2\right)^{2n+1}-\left(x^2-y\right)^3 \cdot \left(y+x^2\right)^{2n}+\left(x^2+y\right)^3 \left(y-x^2\right)^{2n}\right)^2
},
\]
from which we can assert that its zero set forms a set of dimension less than or equal 1. We can therefore deduce that such a zero set does not contain any open subsets in $\R^2$, as needed. We expect that this setup could potentially prove \eqref{gen} for the case when $f$ and $g$ are polynomials with distinct degrees both greater than or equal to 2, although we do not pursue this here due to the fact that the elimination method produces a polynomial of very high degree (in all variables) for $F$ and therefore significantly increases the computational complexity of the second derivative~\eqref{eq:test}. 

The squeezing technique was recently used by the first author \cite{ORN2} to obtain growth exponent strictly greater than $3/2$ with only two sets involved. More precisely, it was proven in \cite{ORN2} that
\[
\max \{ |8A-7A|,|5f(A)-4f(A)| \} \gg |A|^{\frac{3}{2} + \frac{1}{54}}
\]
provided that $f$ is a convex function that satisfies an additional technical condition. One can use Pl\"{u}nnecke's Theorem to reduce the number of variables, obtaining
\begin{equation} \label{twosets}
\max \{ |3A-A|,|3f(A)-f(A)| \} \gg |A|^{\frac{3}{2} + c}.
\end{equation}
The function $f(x)=x^3$ satisfies this condition, and so \eqref{twosets} may be regarded as an improvement to Theorem \ref{thm:main1}. Interestingly, the case when $f(x)=x^2$ is not covered by the result in \cite{ORN2}, and the task of proving that
\[
\max \{ |k_1A-k_2A|,|\ell_1A^2-\ell_2A^2| \} \gg |A|^{\frac{3}{2} + c},
\]
for some constants $k_1,k_2,\ell_1,\ell_2 \in \mathbb N$ and $c>0$, remains open.

%%%%%%%%%%%%%%%%%%%%
\section{Proof of Theorem \ref{thm:main2}}

This proof is based on the same elementary ideas that were used in \cite{HRNS}. However, we give a slightly different presentation of the basic idea here, which we hope is a little easier to digest. The idea of using triples was communicated to us by Misha Rudnev, who found this alternative approach to the proof of the main expander result in \cite{HRNS}. We are very grateful to him for sharing his ideas with us.

A two-fold squeezing argument will be used repeatedly in the proof of Theorem \ref{thm:main22}. To limit repetition, we consolidate these applications in the form of the following lemma.

\begin{lemma} \label{lem:squeeze2}
Suppose that $Y= \{ y_1 < y_2 < \dots <y_n \}$ and $Z=\{z_1 < z_2 < \dots < z_n \}$ are sets of real numbers satisfying the property that
\[
y_i < z_i \,\,\,\, \text{for all}\,\,\,\, i\in [n]  \,\,\,\, \text{and}\,\,\,\, z_i < y_j \,\,\,\, \text{for all}\,\,\,\, i<j.
\]
In other words, the ordered set $Y \cup Z$ starts with an element of $Y$ and then alternates between $Y$ and $Z$. 
Write $x_i=z_i-y_i$. Suppose also that
\[
x_i < x_j \,\,\,\, \text{for all}\,\,\,\, i<j.
\]
Let $I \subset [\lfloor n/2 \rfloor-1]$ be an indexing set of integers. Then there exists a subset $I' \subset I$ such that $|I'| \gtrsim |I|$ and
\begin{equation} \label{lemgoal}
|2Y+2Z-2Y-Z| \gg n^2|\{x_{j+1} - x_j : j \in I' \}|.
\end{equation}
\end{lemma}

\begin{proof}

Define
\[
\gamma_j  := x_{j+1} - x_j,
\]
and 
\[
\Gamma  :=\{ \gamma_j : j \leq \lfloor n/2 \rfloor -1 \}.
\]
It is possible that $\gamma_j=\gamma_{j'}$ and $j \neq j'$. We will need to consider these potential multiplicities, and therefore define
\[
r(\gamma)=| \{ j \in [ \lfloor n/2 \rfloor -1 ]: \gamma_j= \gamma \}|.
\]
Observe that
\[
\sum_{ \gamma \in \Gamma} r(\gamma) = \lfloor n /2 \rfloor -1  \gg n.
\]
Therefore, it follows from dyadic pigeonholing that there exists an integer $L$ and a subset $\Gamma' \subset \Gamma$ such that
\[
L \leq r(\gamma) < 2L,\,\,\,\, \forall \, \gamma \in  \Gamma'
\]
and
\begin{equation} \label{dyadic}
n \lesssim |\Gamma'|L \leq n.
\end{equation}
Now define $I' \subset [ \lfloor n/2 \rfloor -1]$ to be the set
\[
I':=\{ j : \gamma_j \in \Gamma' \}.
\]
To put this another way, we have
\[
\Gamma'= \{ \gamma_j : j \in I' \}.
\]
Observe that $|I'| \gtrsim n$. Note that this set $\Gamma'$ is the set which features in the inequality \eqref{lemgoal}, which we are trying to prove. In other words, our goal is to prove that
\begin{equation} \label{newgoal}
|2Y+2Z-2Y-Z| \gg n^2 |\Gamma'|.
\end{equation} 
We will consider elements in $2Y+2Z-2Y-Z$ of the form
\begin{equation} \label{form}
y_k + x_j + \gamma_{\ell}
\end{equation}
such that
\begin{equation} \label{ranges}
n/2 < k \leq n-1, \, \, j, \ell \in I', \,\, \gamma_{\ell} < \gamma_j.
\end{equation}
It follows, from the assumption that $\gamma_{\ell} < \gamma_j$, that $x_j+ \gamma_{\ell} < x_{j+1}$. Moreover, we have
\begin{equation} \label{squeeze1}
x_j <x_j+ \gamma_{\ell} < x_j + \gamma_j= x_j + (x_{j+1} - x_j) = x_{j+1}.
\end{equation}
Also, since $k >n/2 > j+1$, it follows that
\begin{equation} \label{squeeze2}
y_k  < y_k + x_j + \gamma_{\ell} < y_k+ x_{j+1}  <z_k.
\end{equation}
The last inequality is a re-writing of the bound $x_{j+1} < x_k$, which follows from the monotonicity of the $x_j$ and the fact that $j+1 < k$.

\begin{figure}[ht]

% node positions (change here)
\tikzmath{
\maxlength = 15;
\yj = 1;
\firstgap = 1;
\xj = \yj + (\firstgap / 2);
\yjone = \yj + \firstgap + 1;
\secondgap = 3;
\xjone = \yjone + (\secondgap / 2);
\yk = 9;
\thirdgap = 4;
\xk = \yk + (\thirdgap / 2);
\gammaj = \xjone - \xj;
\gammaell = 1.5;
\bluedot = \xj + \gammaell;
\reddot = \yk + \firstgap + \gammaell;
}

\centering

\begin{tikzpicture}

% main structure and half way mark
\foreach \vertpos in {0,1}{
    \draw (0,\vertpos) -- (\maxlength,\vertpos);
    }
\draw[dashed] (\maxlength/2,0.5) -- ++(0,8mm);
\draw[dashed] (\maxlength/2,0.5) -- ++(0,-8mm) node[below]{$\tfrac{n}2$};

% ordered differences (top line)
\foreach \pos/\descr in {\xj/{\color{blue} x_j},\xjone/x_{j+1},\xk/{x_k}}
{
	\draw (\pos,1) -- ++(0,1mm);
	\draw (\pos,1) -- ++(0,-1mm);
	\node[yshift=-5mm] at (\pos,1) {$\descr$};
}

% ordered A and B sets (bottom line)
\foreach \pos/\descr in {\yj/y_j,\yj+\firstgap/z_j,\yjone/y_{j+1},\yjone+\secondgap/z_{j+1},\yk/y_k,\yk+\thirdgap/z_k}
{
	\draw (\pos,0) -- ++(0,1mm);
    \draw (\pos,0) -- ++(0,-1mm);
    \node[yshift=-4mm] at (\pos,0) {$\descr$};
}

% all top braces
\draw [mybrace] (\xj,1) -- node[above, yshift=4mm]{$\gamma_j$} ++(\gammaj,0);
\draw [mybrace, blue] (\yj,0) -- node[above, yshift=2mm]{} ++(\firstgap,0);
\draw [mybrace] (\yjone,0) -- node[above, yshift=2mm]{} ++(\secondgap,0);
\draw [mybrace] (\yk,0) -- node[above, yshift=2mm]{} ++(\thirdgap,0);

% blue dot + corresponding underbrace + label
\draw[fill,blue] (\bluedot,1) circle (2pt);
\draw [mybracedown, blue] (\xj,1) -- node[below, yshift=-2.5mm]{{\color{blue} $\gamma_{\ell}$}} ++(\gammaell,0);
\draw [->, blue] plot [smooth,tension=1] coordinates {(\bluedot-1,1.5)(\bluedot-0.5,1.3)(\bluedot-0.1,1.1)};
\node[above, left] at (\bluedot-1,1.5) {\color{blue} $x_j+\gamma_{\ell}$};

% red dot + corresponding underbrace + label
\draw[fill,red] (\reddot,0) circle (2pt);
\draw [mybracedown,blue] (\yk,0) -- node[below, yshift=-2.5mm]{{\color{blue}$x_j+\gamma_{\ell}$}} ++(\firstgap+\gammaell,0);
\draw [->, red] plot [smooth,tension=1] coordinates {(\reddot-1,0.5)(\reddot-0.5,0.3)(\reddot-0.1,0.1)};
\node[above, left] at (\reddot-1,0.5) {\color{red} $y_k+x_j+\gamma_{\ell}$};

\end{tikzpicture}

\caption{This figure provides a pictorial view of the ``squeezing'' inequalities in \eqref{squeeze1} and \eqref{squeeze2}. The blue and red dots are the objects to be counted.}
\label{fig:squeeze}
\end{figure}
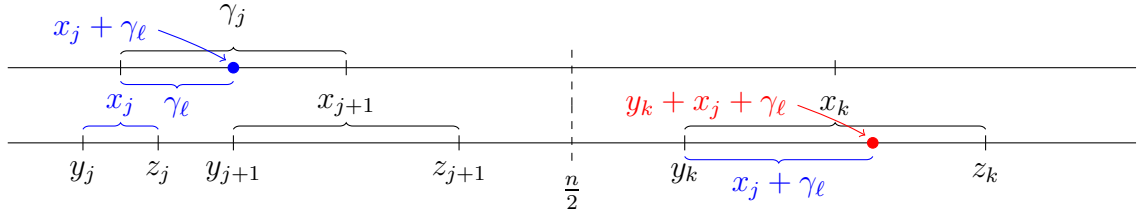

These inequalities allow us to squeeze sums into gaps efficiently. Indeed, let us consider a fixed $k$ in the range given by \eqref{ranges}. Then the sums of the form \eqref{form} belong to the interval $(y_k,z_k)$. As $k$ varies, these intervals are disjoint. Therefore, since the sums of the form \eqref{form} are all elements of the set $2Y+2Z-2Y-Z$, we have
\begin{align*}
2Y+2Z-2Y-Z  &\supseteq \{ y_k + x_j + \gamma_{\ell} : \frac{n}{2} < k \leq n-1, j, \ell \in I', \gamma_{\ell} < \gamma_j\} 
\\& = \bigcup_{\frac{n}{2} < k \leq n-1} \{ y_k + x_j + \gamma_{\ell} :  j, \ell \in I', \gamma_{\ell} < \gamma_j\}.
\end{align*}
This is a disjoint union, and so it follows that
\begin{align} \label{7sums}
&|2Y+2Z-2Y-Z|  \nonumber
\\&\geq \sum_{n/2 < k \leq n-1} |\{x_j + \gamma_{\ell} : j , \ell \in I', \gamma_{\ell} < \gamma_j \}| \gg n  |\{x_j + \gamma_{\ell} : j , \ell \in I', \gamma_{\ell} < \gamma_j \}|.
\end{align}

Moreover, we can use a second squeeze, coming from \eqref{squeeze1}, to obtain a lower bound for the size of the set $\{x_j + \gamma_{\ell} : j, \ell \in I', \gamma_{\ell} < \gamma_j \}$. For fixed $j \in I'$, it follows from \eqref{squeeze1} that 
\[
x_j + \gamma_{\ell} \in (x_j, x_{j+1})
\]
for all $\ell$ such that $\gamma_{\ell}  < \gamma_j$. As $j$ varies, these intervals are disjoint. Therefore,
\begin{equation} \label{sum}
|\{x_j + \gamma_{\ell} : j, \ell \in I', \gamma_{\ell} < \gamma_j \}|= \sum_{j \in I'} |\{ \ell \in I': \gamma_{\ell} < \gamma_j \}|.
\end{equation}
This sum can be seen as the sum over all $j \in I'$ of the rank of $j$, where the rank of $j$ tells us the position of $\gamma_j$ in the ordered set $\Gamma'$. By the dyadic pigeonholing argument used to construct the set $\Gamma'$, at most $2L$ elements $j$ can have the same rank. Therefore, \eqref{sum} is a sum of $|I'|$ non-negative integers, each occuring with multiplicity in between $L$ and $2L$. Since $|I'| \gtrsim n$, it follows that for some absolute constant $c$,
\[
|\{x_j + \gamma_{\ell} : j, \ell \in I_1, \gamma_{\ell} < \gamma_j \}| \geq L \sum_{j=0}^{c n/(L\log n)} j \gtrsim \frac{n^2}{L} \geq n|\Gamma'|.
\]
We have used the upper bound in \eqref{dyadic} for the last inequality. It therefore follows from \eqref{7sums} that
\[
|2Y+2Z-2Y-Z| \gg n^2|\Gamma'|.
\]
This proves \eqref{newgoal}, and therefore completes the proof of the lemma.
\end{proof}

Next, we restate Theorem \ref{thm:main2} in a way that makes for more convenient generalization. 

\begin{theorem} \label{thm:main22}
Let $A \subset \mathbb R$ and define the function $f: \mathbb R \rightarrow \mathbb R$ by $f(x)= \ln( e^x+1)$. Then there exist $a,a' \in A$ such that
\[
|2f(a+A)+2f(a'+A) - 2 f(a+A) - f(a'+A)| \gtrsim |A|^{31/12}.
\]
\end{theorem}

Let us quickly verify that this result does imply Theorem  \ref{thm:main2}. Note that there exists a subset $X'  \subset X$ with $|X'| \geq |X|/2-1$ such that all of the elements of $X'$ have the same sign. If all of the elements of $X'$ are positive then apply Theorem \ref{thm:main22}, taking $A=\ln X'$. Otherwise, all of the elements of $X'$ are negative and we can apply Theorem \ref{thm:main22} with $A=\ln(-X')$. It remains to prove Theorem \ref{thm:main22}. 

\begin{proof}[Proof of Theorem \ref{thm:main22}]

Label the elements of $A$ in ascending order, so $A=\{a_1<\dots < a_n \}$. Identify two ``nextdoor but two" elements of $A$ which are closest. That is, let $a_i$ and $a_{i+3}$ be two elements of $A$ such that
\begin{equation} \label{mindef}
a_{i+3} - a_i = \min \{a_{j+3} - a_j : 1 \leq j \leq n-3 \}.
\end{equation}
Then, for every third element $a_{3j}$ of $A$, where $1 \leq j \leq \lfloor n/3 \rfloor$, we consider the interval
\[
(a_i+a_{3j}, a_{i+3} + a_{3j}).
\]
Note that all of these intervals have the same length, which is $a_{i+3} - a_i$. Moreover, as $j$ varies, these intervals are disjoint. Indeed, suppose that two neighboring such intervals
\[
(a_i+a_{3j}, a_{i+3} + a_{3j}) \,\,\,\text{and} \,\,\,\, (a_i+a_{3(j+1)}, a_{i+3} + a_{3(j+1)})
\]
overlap. Then it follows that
\[
a_i+a_{3(j+1)} <  a_{i+3} + a_{3j}.
\]
This rearranges to give
\[
a_{3j+3} - a_{3j} < a_{i+3} - a_i,
\]
which contradicts the minimality of $a_{i+3}-a_i$ established in \eqref{mindef}.

We now apply the function $f$ to these intervals, so that we consider the intervals
\begin{equation} \label{intervals}
(f(a_i+a_{3j}), f(a_{i+3} + a_{3j})) \,\,\,\,\text{such that} \,\,\, 1 \leq j \leq \lfloor n/3 \rfloor.
\end{equation}
Note that $f$ is strictly convex. It therefore follows that the intervals considered in \eqref{intervals} have lengths which are strictly increasing as $j$ increases.

We will make use of three applications of Lemma \ref{lem:squeeze2}. To this end, we define
\begin{align*}
c_j &:= f(a_{i+1} + a_{3j}) -  f(a_{i} + a_{3j}), \\ 
d_j &:= f(a_{i+2} + a_{3j}) -  f(a_{i+1}+a_{3j}), \\ 
e_j &:= f(a_{i+3} + a_{3j}) -  f(a_{i+2} + a_{3j}). 
\end{align*}

See Figure \ref{fig:squeeze2} for a visual explanation of the sequences for these intervals $\{c_j\}_{j\geq 1}$, $\{d_j\}_{j\geq 1}$, and $\{e_j\}_{j\geq 1}.$ In Figure \ref{fig:squeeze2}, for the purposes of making the annotations readable, we use $a_{k,j}$ as a temporary shorthand to represent $a_{i+k}+a_{3j}$.

\begin{figure}
\centering

% node positions (change here)
\tikzmath{
\maxlength = 20;
\topline = 3;
\bottomline = 0;
\midline = (\topline + \bottomline)/2;
\topfromtick = \topline - 0.1;
\toptotick = \topline + 0.1;
\bottomfromtick = \bottomline - 0.1;
\bottomtotick = \bottomline + 0.1;
\firstgap1 = 1.4;
\firstgap2 = 2.4;
\firstgap3 = 1.8;
\aiaj = 0.2;
\aioneaj = \aiaj + \firstgap1 ;
\aitwoaj = \aioneaj + \firstgap2 ;
\aithreeaj = \aitwoaj + \firstgap3 ;
\aiajone = \aiaj + (\maxlength / 2.2);
\aioneajone = \aiajone + \firstgap1;
\aitwoajone = \aioneajone + \firstgap2;
\aithreeajone = \aitwoajone + \firstgap3;
\expandj = 0.3;
\secondgap1 = \firstgap1 + \expandj;
\secondgap2 = \firstgap2 + \expandj;
\secondgap3 = \firstgap3 + \expandj;
\faiaj = 0.4;
\faioneaj = \faiaj + \secondgap1;
\faitwoaj = \faioneaj + \secondgap2;
\faithreeaj = \faitwoaj + \secondgap3;
\cj = \faiaj + (\secondgap1 / 2);
\dj = \faioneaj + (\secondgap2 / 2);
\ej = \faitwoaj + (\secondgap3 / 2);
\expandj1 = 1.3;
\thirdgap1 = \firstgap1 + \expandj1;
\thirdgap2 = \firstgap2 + \expandj1;
\thirdgap3 = \firstgap3 + \expandj1;
\faiajone = \faiaj + (\maxlength / 2);
\faioneajone = \faiajone + \thirdgap1;
\faitwoajone = \faioneajone + \thirdgap2;
\faithreeajone = \faitwoajone + \thirdgap3;
\cj1 = \faiajone + (\thirdgap1 / 2);
\dj1 = \faioneajone + (\thirdgap2 / 2);
\ej1 = \faitwoajone + (\thirdgap3 / 2);
\fleftfrom = \aioneaj + (\firstgap2)/2 ;
\fleftto = \faioneaj+(\secondgap2)/2;
\fleftmid = (\fleftfrom + \fleftto)/2;
\frightfrom = \aioneajone + (\firstgap2)/2 ;
\frightto = \faioneajone + (\thirdgap2)/2;
\frightmid = (\frightfrom + \frightto)/2;
}

\begin{tikzpicture}[scale=0.75]

% main structure
\foreach \vertpos in {\bottomline,\topline}{
    \draw (0,\vertpos) -- (\maxlength,\vertpos);
    }
    
% ordered A and B sets
\foreach \pos/\descr in {
	\aiaj/a_{0,j},
	\aioneaj/a_{1,j},
	\aitwoaj/a_{2,j},
	\aithreeaj/a_{3,j},
	\aiajone/a_{0,j+1},
	\aioneajone/a_{1,j+1},
	\aitwoajone/a_{2,j+1},
	\aithreeajone/a_{3,j+1}
}
{
	\draw (\pos,\topfromtick) -- (\pos,\toptotick);
    \node[yshift=-4mm] at (\pos,\topline) {$\descr$};
}

% ordered f(A) and f(B) sets
\foreach \pos/\descr in {
	\faiaj/f(a_{0,j}),
	\faioneaj/f(a_{1,j}),
	\faitwoaj/f(a_{2,j}),
	\faithreeaj/f(a_{3,j}),
	\faiajone/f(a_{0,j+1}),
	\faioneajone/f(a_{1,j+1}),
	\faitwoajone/f(a_{2,j+1}),
	\faithreeajone/f(a_{3,j+1})
}
{
	\draw (\pos,\bottomfromtick) -- (\pos,\bottomtotick);
	\node[yshift=-4mm] at (\pos,\bottomline) {$\descr$};
}

% red
\foreach \line/\pos/\gap in {
	\topline/\aiaj/\firstgap1,
	\topline/\aiajone/\firstgap1,
	\bottomline/\faiaj/\secondgap1,
	\bottomline/\faiajone/\thirdgap1
}
	\draw[red, ultra thick] (\pos,\line) -- (\pos+\gap,\line);

% blue
\foreach \line/\pos/\gap in {
	\topline/\aioneaj/\firstgap2,
	\topline/\aioneajone/\firstgap2,
	\bottomline/\faioneaj/\secondgap2,
	\bottomline/\faioneajone/\thirdgap2
}
	\draw[blue, ultra thick] (\pos,\line) -- (\pos+\gap,\line);

% yellow
\foreach \line/\pos/\gap in {
	\topline/\aitwoaj/\firstgap3,
	\topline/\aitwoajone/\firstgap3,
	\bottomline/\faitwoaj/\secondgap3,
	\bottomline/\faitwoajone/\thirdgap3
}
	\draw[yellow, ultra thick] (\pos,\line) -- (\pos+\gap,\line);
	
% arrows for applying f
\draw[->] (\fleftfrom,\topline-0.5) -- (\fleftto,\bottomline+1.2);
\node at (\fleftmid+0.6, \midline+0.3) {$f$};
\draw[->] (\frightfrom,\topline-0.5) -- (\frightto,\bottomline+1.2);
\node[right] at (\frightmid+1.1, \midline+0.3) {$f$};

% all top braces
\draw [mybrace]
(\faiaj + 0.05, \bottomline) -- node[above, yshift=3mm]{$c_j$}
++(\secondgap1 - 0.1, \bottomline);
\draw [mybrace]
(\faioneaj + 0.05,\bottomline) -- node[above, yshift=3mm]{$d_j$}
++(\secondgap2 - 0.1, \bottomline);
\draw [mybrace]
(\faitwoaj + 0.05,\bottomline) -- node[above, yshift=3mm]{$e_j$}
++(\secondgap3 - 0.1, \bottomline);
\draw [mybrace]
(\faiajone + 0.05,\bottomline) -- node[above, yshift=3mm]{$c_{j+1}$}  
++(\thirdgap1 - 0.1, \bottomline);
\draw [mybrace]
(\faioneajone + 0.05,\bottomline) -- node[above, yshift=3mm]{$d_{j+1}$}  
++(\thirdgap2 - 0.1, \bottomline);
\draw [mybrace]
(\faitwoajone + 0.05,\bottomline) -- node[above, yshift=3mm]{$e_{j+1}$}  
++(\thirdgap3 - 0.1, \bottomline);

\end{tikzpicture}

\caption{This figure illustrates the setup for the proof of Theorem \ref{thm:main22}. In the proof, we begin by finding many identical intervals, divided into three parts (color-coded in the figure for convenience), with endpoints in $A+A$ which do not overlap. This is illustrated in the top line. Application of our strictly convex function $f$ to these intervals changes their lengths. This is illustrated in the second line. After the application, the lengths of the colored intervals form an increasing sequence as $j$ increases. This puts us in a situation where we can look to apply Lemma \ref{lem:squeeze2} for each of these sets of increasing intervals.}
\label{fig:squeeze2}
\end{figure}
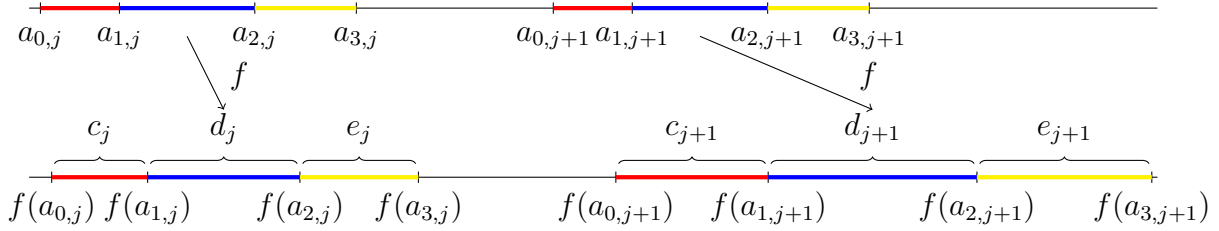

In words, $c_j$ is the difference between the first and second points of the $j$'th interval identified in \eqref{intervals},  $d_j$ is the difference between the second and third points of the $j$'th interval, and  $e_j$ is the difference between the third and fourth points of the $j$'th interval. Since $f$ is strictly convex, each of these quantities are strictly increasing with $j$. That is, we have
\[
c_j < c_{j+1}, \,\, d_j < d_{j+1}, \, \, e_j < e_{j+1}
\]
for all $1 \leq j \leq \lfloor n/3 \rfloor -1$.

Apply Lemma \ref{lem:squeeze2} with

\[
Y=\{f(a_i+a_{3j} ) : 1 \leq  j  \leq \lfloor n/3 \rfloor \}, \,\,\,\, Z=\{f(a_{i+1}+a_{3j} ) : 1 \leq j \leq  \lfloor n/3 \rfloor \},
\] 
\[
x_j=c_j, \,\,\,\, \text{and} \,\,\, I=[\lfloor n/6\rfloor -1].
\]
It follows that there exists $I_1 \subset [\lfloor n/6\rfloor  -1]$ such that
\[
|I_1| \gtrsim n
\]
and
\begin{equation} \label{conc1}
|2f(a_i+A) + 2f(a_{i+1}+A) - 2f(a_i+A) - f(a_{i+1}+A)|  \gg n^2 | \{ c_{j+1} - c_j : j \in I_1 \}|.
\end{equation}

Next, we make a second application of Lemma \ref{lem:squeeze2}. This time, we set
\[
Y=\{f(a_{i+1}+a_{3j} ) : 1 \leq  j  \leq \lfloor n/3 \rfloor \}, \,\,\,\, Z=\{f(a_{i+2}+a_{3j} ) : 1 \leq j \leq  \lfloor n/3 \rfloor \}, 
\]
\[
x_j=d_j, \,\,\,\, \text{and} \,\,\, I=I_1.
\]
It follows that there exists $I_2 \subset I_1$ such that
\[
|I_2| \gtrsim |I_1| \gtrsim n
\]
and
\begin{equation} \label{conc2}
|2f(a_{i+1}+A) + 2f(a_{i+2}+A) - 2f(a_{i+1}+A) - f(a_{i+2}+A) | \gg n^2 | \{ d_{j+1} - d_j : j \in I_2 \}|.
\end{equation}

For the third application of Lemma \ref{lem:squeeze2}, set
\[
Y=\{f(a_{i+2}+a_{3j} ) : 1 \leq j  \leq \lfloor n/3 \rfloor \}, \,\,\,\, Z=\{f(a_{i+3}+a_{3j} ) : 1 \leq j \leq  \lfloor n/3 \rfloor \}, 
\]
\[
x_j=e_j, \,\,\,\, \text{and} \,\,\, I=I_2.
\]
It follows that there exists $I_3 \subset I_2 \subset I_1$ such that
\[
|I_3| \gtrsim |I_2| \gtrsim n
\]
and
\begin{equation} \label{conc3}
|2f(a_{i+2}+A) + 2f(a_{i+3}+A) - 2f(a_{i+2}+A) - f(a_{i+3}+A) | \gg n^2 | \{ e_{j+1} - e_j : j \in I_3 \}|.
\end{equation}

Next, define
\[
\Gamma := \{ c_{j+1}-c_j : j \in I_3 \}, \,\,\,\,\,\,\, \Delta := \{ d_{j+1} - d_j : j \in I_3 \}, \,\,\,\, \mathcal E := \{ e_{j+1} - e_j : j \in I_3 \}.
\]

It now follows from \eqref{conc1}, \eqref{conc2} and \eqref{conc3}, along with the inclusions $I_3 \subset I_2 \subset I_1$ that there exist $a,a' \in A$ such that
\begin{equation} \label{concsum}
|2f(a+A) + 2f(a'+A) - 2f(a +A) -f(a'+A) | \gtrsim n^2 \max \{|\Gamma|, |\Delta|,  |\mathcal E| \}.
\end{equation}

The last task is for us to prove the following claim.

\begin{claim}
\[
\max \{|\Gamma|, |\Delta|,  |\mathcal E| \} \gg n^{7/12}.
\]
\end{claim}
Once this claim is proven, the proof of the theorem will follow from \eqref{concsum}. It remains to prove the claim.

\begin{proof}[Proof of Claim]

We will double count solutions to the system of equations
\begin{align} \label{system}
\gamma &=c_{j+1} - c_j, \nonumber
\\ \delta &= d_{j+1} - d_j,
\\ \varepsilon &= e_{j+1} - e_j, \nonumber
\end{align}
such that
\[
\gamma \in \Gamma, \,\, \delta \in \Delta, \,\, \varepsilon \in \mathcal E, \,\, j \in I_3.
\]
Let $S$ denote the number of solutions to this system. We note here that we do not get any solutions to this system with any of $\gamma$, $\delta$ or $\varepsilon$ equal to $0$. Indeed, this follows from our knowledge that $c_{j+1} > c_j$, $d_{j+1} > d_j$ and $e_{j+1} > e_j$. 

The easy task is to give a lower bound. Indeed, since each $j \in I_3$ gives a contribution to $S$, we have
\begin{equation} \label{lower}
S = |I_3| \gtrsim n.
\end{equation}
To obtain an upper bound for $S$, we will unravel the definitions which are used in \eqref{system}. We remark here that, until this point, we have not used the specified information that $f(x)= \ln (e^x+1)$. The only property of this function that we have used is that $f$ is strictly convex.

Plugging the definitions of $c_j,d_j$ and $e_j$ into \eqref{system}, we can rewrite it in the form
\begin{align} \label{system2}
\gamma&= f(a_{i+1} + a_{3(j+1)}) -  f(a_{i} + a_{3(j+1)}) - f(a_{i+1} + a_{3j}) +  f(a_{i} + a_{3j}), \nonumber
\\ \delta &= f(a_{i+2} + a_{3(j+1)}) -  f(a_{i+1} + a_{3(j+1)}) - f(a_{i+2} + a_{3j}) +  f(a_{i+1} + a_{3j}),
\\ \varepsilon &= f(a_{i+3} + a_{3(j+1)}) -  f(a_{i+2} + a_{3(j+1)}) - f(a_{i+3} + a_{3j}) +  f(a_{i+2} + a_{3j}). \nonumber
\end{align}
Recall that the elements $a_i,a_{i+1}, a_{i+2}$ and $ a_{i+3}$ are not variables here. They were fixed at the beginning of the proof when we chose close near-neighbors $a_i$ and $a_{i+3}$. Write $a_k= \ln(s_k)$ for all $1 \leq k \leq n$. We can finally use the information that $f(x)= \ln (e^x+1)$ to rewrite this system once again. After some rearranging, we obtain
\begin{align} \label{system3}
e^{\gamma}&=\frac{ (s_{i+1}s_{3(j+1)}+1)(s_{i}s_{3j}+1)}{ (s_is_{3(j+1)}+1)(s_{i+1}s_{3j}+1)},\nonumber
\\ e^{\delta}&=\frac{ (s_{i+2}s_{3(j+1)}+1)(s_{i+1}s_{3j}+1)}{ (s_{i+1}s_{3(j+1)}+1)(s_{i+2}s_{3j}+1)},
\\ e^{\varepsilon}&= \frac{ (s_{i+3}s_{3(j+1)}+1)(s_{i+2}s_{3j}+1)}{ (s_{i+2}s_{3(j+1)}+1)(s_{i+3}s_{3j}+1)}. \nonumber
\end{align}
After relabeling the variables, we seek an upper bound for the number of solutions to the system
\begin{align} \label{system33}
\bar{x}&=\frac{ (s_{i+1}s_{3(j+1)}+1)(s_{i}s_{3j}+1)}{ (s_is_{3(j+1)}+1)(s_{i+1}s_{3j}+1)},\nonumber
\\ \bar{y} &=\frac{ (s_{i+2}s_{3(j+1)}+1)(s_{i+1}s_{3j}+1)}{ (s_{i+1}s_{3(j+1)}+1)(s_{i+2}s_{3j}+1)},
\\ \bar{z} &= \frac{ (s_{i+3}s_{3(j+1)}+1)(s_{i+2}s_{3j}+1)}{ (s_{i+2}s_{3(j+1)}+1)(s_{i+3}s_{3j}+1)}, \nonumber
\end{align}
 such that
\[
\bar{x} \in e^{\Gamma}, \,\, \bar{y} \in e^{\Delta}, \,\, \bar{z} \in e^{\mathcal E}, \,\, j \in I_3.
\]
Since we do not have any solutions to \eqref{system} with any of $\gamma, \delta, \varepsilon = 0$, it follows that there are no solutions to \eqref{system33} with any of $\bar{x},\bar{y}$ or $\bar{z}$ equal to $1$.

For the goal of finding an upper bound to the number of solutions to \eqref{system33}, we use the Elekes-Szab\'{o} Theorem in much the same way as the proof of Lemma \ref{lem:main}. We first use Singular's \textsc{Eliminate} \cite{DGPS, KL} to compute the elimination ideal generated by the polynomials
\begin{align*}
(s_is_{3(j+1)}+1)(s_{i+1}s_{3j}+1)x&-(s_{i+1}s_{3(j+1)}+1)(s_{i}s_{j}+1),\\
(s_{i+1}s_{3(j+1)}+1)(s_{i+2}s_{3j}+1)y&-(s_{i+1}s_{3(j+1)}+1)(s_{i+2}s_{3j}+1),\\
(s_{i+2}s_{3(j+1)}+1)(s_{i+3}s_{3j}+1)z&-(s_{i+3}s_{3(j+1)}+1)(s_{i+2}s_{3j}+1),
\end{align*}
which eliminates $s_{3j}$ and $s_{3(j+1)}$ from this system and reduces it to a single polynomial in three variables $x, y$ and $z$. From there, we deduce that every contribution $(x,y,z,j)$ to $S$ gives rise to a solution to the polynomial equation $F(x,y,z)=0$, where 
\begin{equation} \label{Fdefn}
F(x,y,z):= \,(x-1)(z-1)\cdot G(x,y,z),
\end{equation} 
with $G$ defined to be
\begin{align*} \label{Gdefn}
G(x,y,z):= \, & (x y^2 z+1) (s_i-s_{i+1})(s_{i+2}-s_{i+3})\\ \nonumber
& + (xyz + y) (s_{i+2}-s_{i})(s_{i+1}-s_{i+3})\\ \nonumber
& + (yz + xy)(s_{i}-s_{i+3})(s_{i+1}-s_{i+2}).
\end{align*}
Since there are no contributions to $S$ with $x=1$ or $z=1$, it follows that every contribution $(x,y,z,j)$ to $S$ gives rise to a solution to the polynomial equation $G(x,y,z)=0$. Therefore, we have
\begin{equation} \label{middle}
S \leq |Z(G) \cap ( e^{\Gamma} \times e^{\Delta} \times e^{\mathcal E} )|.
\end{equation}
Note that all of the coefficients in $G$ are non-zero, since $s_i<s_{i+1}<s_{i+2}<s_{i+3}$. We now proceed to verify that $G$ is non-degenerate. In order to see this, we note that  $z$ can be expressed as a rational function of $x$ and $y$, i.e.,
\[
z=\frac{-x y(s_{i}-s_{i+3})(s_{i+1}-s_{i+2})-y(s_{i+2}-s_{i})(s_{i+1}-s_{i+3})-(s_i-s_{i+1})(s_{i+2}-s_{i+3})}
{xy^2 (s_i-s_{i+1})(s_{i+2}-s_{i+3})+xy(s_{i+2}-s_{i})(s_{i+1}-s_{i+3})+y(s_{i}-s_{i+3})(s_{i+1}-s_{i+2})}.
\]

Denote the right hand side of this expression to be $g(x,y)$ (note that this is $f(x,y)$ from Lemma \ref{lem:difftest} but we changed the name to distinguish it from the $f$ in this proof). Then the computation of \[\frac{\partial^2(\ln|g_x/g_y|)}{\partial x \partial y}\] turns out to be the rational function
\begin{equation}\label{eq:partial2g}
\dfrac{\begin{aligned}
-2\cdot \big(&\, (x^2 y^2 + 1)\left(s_i-s_{i+1}\right) \left(s_{i+1}-s_{i+2}\right) \left(s_i-s_{i+3}\right)\\
&+\, (x^2 y + y)\left(s_i-s_{i+2}\right) \left(s_{i+1}-s_{i+2}\right) \left(s_i-s_{i+3}\right) \left(s_{i+1}-s_{i+3}\right)\\
&+\, (2 x y) \left(s_{i+1}-s_{i+2}\right)^2 \left(s_i-s_{i+3}\right)^2
\big)\\[3pt]
\end{aligned}}
{
\begin{aligned}
\big(\,
(x^2 y^2 + 1)&(s_{i+1}-s_{i+2})(s_i-s_{i+3})
 - (x y^2 + x) (s_{i}-s_{i+2})(s_{i+1}-s_{i+3})\\
+\, (2 xy)&(s_{i}-s_{i+1})(s_{i+2}-s_{i+3})
- x (s_{i}-s_{i+2})(s_{i+1}-s_{i+3})
\,\big){}^2
\end{aligned}
}\, ,
\end{equation}
which is an expression that is zero when
\begin{equation}\label{eq:solvedegenerate}
y=\begin{cases}
r_1\pm \frac12 \sqrt{r_2}, & x\neq 0,\\
\dfrac{(s_{i}-s_{i+1})(s_{i+2}-s_{i+3})}{(s_{i}-s_{i+2})(s_{i+1}-s_{i+3})}, & x=0,
\end{cases}
\end{equation}
with
\vspace{3pt}

\begin{align*}
r_1&:=\frac{(x^2+1) \left(s_i-s_{i+2}\right) \left(s_{i+1}-s_{i+3}\right)-2x \left(s_{i+1}-s_{i+2}\right) \left(s_i-s_{i+3}\right)}{2x^2  \left(s_i-s_{i+1}\right) \left(s_{i+2}-s_{i+3}\right)},\\[10pt]
r_2&:=\frac{\begin{aligned}
& (x^2+1)\left(s_i s_{i+1}-s_{i+2} s_{i+1}+s_{i+2} s_{i+3}\right) \left(s_i s_{i+1}-s_{i+2} s_{i+1}-2s_i s_{i+3}+s_{i+2} s_{i+3}\right)\\
-&2x \left(s_i-s_{i+2}\right) \left(s_{i+1}-s_{i+3}\right) \left(s_i s_{i+1}+s_{i+2} s_{i+1}-2 s_{i+3} s_{i+1}-2 s_i s_{i+2}+s_i s_{i+3}+s_{i+2} s_{i+3}\right)
\end{aligned}}
{x^4 \left(s_i-s_{i+1}\right)^2 \left(s_{i+2}-s_{i+3}\right)^2}.
\end{align*}

The pairs $(x,y)$ that satisfy \eqref{eq:solvedegenerate} form a set of measure zero in $\R^2$. Hence, there are no open subsets in $\R^2$ for which \eqref{eq:partial2g} is identically zero. By Lemma \ref{lem:difftest}, we can conclude that $G$ is non-degenerate.

Since $G$ is non-degenerate, it follows from Theorem \ref{thm:11/6} that
\[
|Z(G) \cap (e^{\Gamma} \times e^{\Delta} \times e^{\mathcal E}) | \ll \max \{ |\Gamma|,|\Delta|,|\mathcal E| \}^{12/7}.
\]
Combining this with \eqref{middle} and \eqref{lower}, it follows that
\[
\max \{ |\Gamma|, |\Delta|, |\mathcal E| \} \gtrsim n^{7/12},
\]
as required.

\end{proof}

This completes the proof of Theorem \ref{thm:main22}.
\end{proof}

%%%%%%%%%%%%%%%%%%%%
\section{Conclusions and Future Directions}

We conclude this paper by discussing a few variations of Theorem~\ref{thm:main22}, in which we make different choices for the function $f$. We observe that most of the proof uses only the fact that $f$ is a strictly convex function. However, we eventually need to know something more concrete about $f$ in order to calculate $F$ in the form of \eqref{Fdefn}, and to check whether or not it is degenerate.

Firstly, we note that we can repeat the proof of Theorem \ref{thm:main22} with minimal changes for the case when $f(x)=\ln(e^x+\lambda)$ and $\lambda$ is a strictly positive real number. That is, we have the following theorem.

\begin{theorem} \label{thm:main222}
Let $A \subset \mathbb R$ and let $\lambda>0$ be a real number. Define the function $f: \mathbb R \rightarrow \mathbb R$ by $f(x)= \ln( e^x+\lambda)$. Then there exist $a,a' \in A$ such that
\[
|2f(a+A)+2f(a'+A) - 2 f(a+A) - f(a'+A)| \gtrsim |A|^{31/12}.
\]
\end{theorem}
This in turn implies an analog of Theorem \ref{thm:main2}: for any $X \subset  \mathbb R$ and any $\lambda>0$, there exist $x,x' \in X$ such that
\[
\left | \frac{(xX+\lambda)^{(2)}(x'X+\lambda)^{(2)}}{(xX+\lambda)^{(2)}(x'X+\lambda)} \right | \gtrsim |X|^{31/12}.
\]
We can also prove a similar result for $\lambda < 0$, although a little more care is needed to ensure we are not attempting to take logarithms of negative values. 

We attempted to prove a version of Theorem~\ref{thm:main22} with $f(x)=x^3$ and $f(x)=x^4$. For the case $f(x)=x^3$, we discovered that the corresponding polynomial $F(x,y,z)$ is linear in each of the three variables and thus degenerate, so Theorem \ref{thm:11/6} could not be applied. For the case when $f(x)=x^4$, we calculated that the corresponding polynomial $F(x,y,z)$ is quartic in all 3 variables, and we were not able to verify that this polynomial is non-degenerate using Lemma~\ref{lem:test} as \eqref{eq:test} turned out to be too computationally complex. Nevertheless, we expect that such results for $f(x)=x^n$ with $n\geq 3$ are true. 

These cases illustrate a general problem with the method: on one hand, we have a useful derivative test in the form of Lemma \ref{lem:test} to check whether a given polynomial $F(x,y,z)$ satisfies the non-degeneracy conditions of Theorem~\ref{thm:11/6}. On the other hand, when it either gets too computationally complex (as it often does) or we get no useful information because the second derivative is identically zero, the use of alternative strategies tend to be more ad-hoc and adapted according to the underlying geometric properties of the problem at-hand. Thus, a reasonable direction of research could be to identify another systematic way in which we could determine non-degeneracy of polynomials in the sense of Definition~\ref{def:nondegen} for the purpose of taking advantage of the result of Elekes and Szab\'{o}.

%%%%%%%%%%%%%%%%%%%%%
\section*{Acknowledgements}

O.~Roche-Newton was supported by the Austrian Science Fund FWF Project P34180. E.~Wong was mostly supported by the Austrian Academy of Sciences at the Radon Institute for Computational and Applied Mathematics (RICAM) during the writing and preparation of this article. We are grateful to Orit Raz and Audie Warren for helpful input when putting together this paper.

%%%%%%%%%%%%%%%%%%%%%
% BIBLIOGRAPHY

\end{document}